\documentclass[a4paper, 12pt]{amsart}
\usepackage[T1]{fontenc}
\usepackage{amssymb,amsthm,amsmath}
\usepackage{xspace}
\usepackage{enumerate}

\newcommand{\R}{\mathbb{R}}

\newcommand{\cF}{\mathcal{F}}

\newcommand{\cN}{\mathcal{N}}

\newtheorem{theorem}{Theorem}[section]
\newtheorem{lemma}[theorem]{Lemma}
\newtheorem{proposition}[theorem]{Proposition}

\theoremstyle{remark}
\newtheorem{remark}[theorem]{Remark}

\newtheorem{definition}[theorem]{Definition}

\begin{document}

\title[]
{Stability of travelling waves in stochastic Nagumo equations}
\author{Wilhelm Stannat}
\address{%
Institut f\"ur Mathematik\\
Technische Universit\"at Berlin \\
Stra{\ss}e des 17. Juni 136\\
D-10623 Berlin\\
and\\
Bernstein Center for Computational Neuroscience\\
Philippstr. 13\\
D-10115 Berlin\\
Germany}
\email{stannat@math.tu-berlin.de}
\date{Berlin, December 12, 2013}

\begin{abstract}
Stability of travelling waves for the Nagumo equation on the whole line is proven using a new approach via functional inequalities 
and an implicitely defined phase adaption. The approach can be carried over to obtain the stability of travelling wave solutions 
in the case of the stochastic Nagumo equation as well. The noise term considered is of multiplicative type with trace-class covariance. 
\end{abstract}


\maketitle 

\section{Introduction} 

\noindent 
The purpose of this paper is to introduce a new approach to the study of (local) stability of travelling waves and pulses in 
excitable media that is in particular well-suited for stochastic perturbations. We are interested in the classical equations 
modelling the propagation of the action potential travelling along the axon of a neuron. As a starting point in this paper 
we consider the Nagumo equation on the real line (cf. \cite{Nag}) perturbed by stochastic forcing terms. We make particular 
use of the explicit knowledge of the travelling waves in this case. However, our approach will be robust w.r.t. small 
perturbations in the coefficients.   

\smallskip 
\noindent 
Since the spectral considerations, employed in the classical stability analysis of nerve axon equations (cf. \cite{Ev, Henry, Jones} 
and the recent monograph \cite{ET}) are not easy to carry over to the stochastic case, we look for a pathwise stability analysis 
in the sense of the classical Lyapunov approach to the stability of dynamical systems. A first novelty of the paper is the introduction 
of an additional dynamics of gradient type that adapts a given solution of the stochastic Nagumo equation to the correct phase of the 
travelling wave. This explicitely given phase adaption, which is in addition easy to implement numerically, is the analogue of the 
phase conditions introduced as algebraic constraints in the classical stability analysis (see in particular \cite{Henry}). As a second 
novelty in this paper, we replace the usual spectral considerations, applied to the Schr\"odinger operator, obtained as linearization 
of the underlying dynamics along a given travelling wave, by functional inequalities of Poincare type. Our hope is that the latter method 
will be generalizable also to general systems of reaction diffusion type because it only uses partial information of the travelling wave 
solutions. Certainly, it is well suited for stochastic perturbations as demonstrated in this paper. An additional advantage is that, in 
contrast to the usual spectral considerations, our approach allows explicit quantitative estimates, both, in the deterministic and in 
the stochastic case and sensitivity considerations w.r.t. the coefficients. 

\smallskip 
\noindent 
The paper is organized as follows: In Section \ref{sec1} we first present our new approach in the case of the deterministic Nagumo 
equation, to demonstrate the main arguments in a somewhat easier setting. The analogue to the usual spectral considerations of the 
Schr\"odinger operator associated with the linearization along a travelling wave is contained in Theorem \ref{thm0}. Our result obtained 
on the spectral gap is optimal (see Proposition \ref{PropPoincare}). Theorem \ref{th1} then contains our main result on the local stability of 
travelling wave solutions. In Section \ref{sec2} we consider the Nagumo equation perturbed with multiplicative noise. Combining our 
stability analysis of Section \ref{sec1} with a careful analysis of the stochastic perturbation, we obtain in Theorem \ref{th2} the 
stochastic analogue of our local stability result in the deterministic case. Our identification of the implicitely defined phase allows to 
rigorously set up a stochastic differential equation for the speed of the wave front and thus gives rise to the correct decomposition 
of the stochastic dynamics into the travelling wave and random fluctuations. Work is in progress to generalize the approach to the stochastic 
neural fields equations considered in \cite{BW}.

\smallskip 
\noindent 
In addition to the new approach to the stability analysis via functional inequalities we also would like to mention that the type of 
stochastic Nagumo equations considered in this paper are also new in comparison with the models of spatially extended neurons subject 
to noise studied numerically and analytically by Tuckwell and Jost in \cite{T2010, T2011} and also by Lord and Th\"ummler in 
\cite{LT}. In order to ensure existence and uniqueness of a solution to our stochastic partial differential equation we use 
the variational approach to stochastic evolution equations as presented in monograph \cite{PR} with recent extensions presented in 
\cite{LR}. In particular, we make use of the Ito formula, that can be obtained for the Hilbert space norm of the variational solution. 
The implied semimartingale decomposition can then be used to apply the one-dimensional (time-dependent) Ito formula to any smooth 
transformation of the Hilbert norm.

\section{The deterministic case} 
\label{sec1}

\medskip 
\noindent 
Consider the Nagumo equation 
\begin{equation} 
\label{Nag} 
\begin{aligned} 
\partial_t v  (t,x) & = \nu\partial_{xx}^2 v (t,x) + bf(v(t,x)) \quad
(t,x) \in \Bbb R_+ \times \Bbb R 
\end{aligned} 
\end{equation} 
on the real line, with $\nu$, $b >0$ and   
$$ 
f(v) = v(1-v)(v-a)  \quad a \in \left( 0, 1\right)\, . 
$$ 
The equation is obtained from the well-known Fitz-Hugh Nagumo system 
\begin{equation} 
\label{FHN} 
\begin{aligned} 
\partial_t v (t,x) & = \nu\partial_{xx}^2 v (t,x) + bf(v(t,x)) - w(t,x) + I \\
\partial_t w (t,x) & = \varepsilon\left( v(t,x) - \gamma w(t,x)\right)  \quad (t,x) \in \Bbb R_+ \times \Bbb R 
\end{aligned} 
\end{equation} 
by letting $\varepsilon\downarrow 0$, i.e., setting the recovery variable $w$ constant, and further equal to the 
input current $I$. It is well-known that for parameters in the exitable region, the Fitz-Hugh Nagumo system admits a travelling pulse solution modelling 
signal propagation along the axon of a single neuron. The analogue for the Nagumo equation is a travelling wave front $v (t,x) = v^{TW} (x + ct)$, where 
\begin{equation} 
\label{TWsolution} 
v^{TW} (x) = \left( 1+\exp\left( -\sqrt{\frac b{2\nu}} x\right)\right)^{-1} 
\end{equation} 
moving to $- \infty$ at constant speed $c = \sqrt{2\nu b}\left( \frac 12 - a\right)$ (cf. \cite{CG}). We are interested in the local stability of this wave front  
in the function space $H = L^2 (\Bbb R)$.

\medskip 
\noindent 
Before we can state a precise definition of stability, we need to introduce first our concept of a solution that we are working with. 
To simplify notations in the following we write $v^{TW} (t) = v^{TW} (\cdot + ct)$. Next (formally) decompose the function $v (t, \cdot )  
= u(t, \cdot ) + v^{TW} (t)$ w.r.t. the travelling wave. The resulting equation for $u$ is then given by 
\begin{equation} 
\label{RelNag} 
\begin{aligned} 
\partial_t u (t,x) & = \nu\partial_{xx}^2 u (t,x) + b\left( f(u(t,x) + v^{TW} (t)) - f(v^{TW} (t))\right) \\
(t,x) \in \Bbb R_+ \times \Bbb R \, . 
\end{aligned} 
\end{equation} 

\medskip 
\noindent 
For the precise definition of the Laplacian $\partial^2_{xx}$ we need to introduce the Sobolev space $V = H^{1,2} (\R)$ of order $1$, 
equipped with the usual norm $\|u\|_V^2 := \int\left( \partial_x u\right)^2\, dx + \|u\|_H^2$. Clearly, $V\hookrightarrow H$ 
densely and continuously. Identifying $H$ with its dual $H^\prime$ we obtain the embeddings $V\hookrightarrow H\equiv H^\prime 
\hookrightarrow V^\prime$. Recall that w.r.t. this embedding the dualization $_{V^\prime}\langle f, u\rangle_V$ between $f\in V^\prime$ and $u\in V$ 
reduces to $_{V^\prime}\langle f, u\rangle_V = \langle f, v\rangle_H = \int fu\, dx$, i.e. the scalar product in $H$ in the case where $f\in H$. 
The Laplacian $\partial^2_{xx}$ then induces a linear continuous mapping $\Delta : V\rightarrow V^\prime$, since $_{V^\prime}\langle \Delta u, v\rangle_V 
= - \int \partial_x u \partial_x v\, dx \le \|u\|_V \|v\|_V$. 

\medskip 
\noindent 
The nonlinear term 
\begin{equation} 
\label{G}
G(t,u) := f(u(t,x) + v^{TW} (t)) - f(v^{TW} (t)) 
\end{equation}  
in equation \eqref{RelNag} can be realized as a continuous mapping 
$$ 
G (\cdot, \cdot ): [0, \infty ] \times V \to V^\prime 
$$ 
that is Lipschitz w.r.t. the second variable on bounded subsets of $V$ with Lipschitz constant independent of $t$. Indeed, due to the elementary 
estimate $\|u\|_\infty \le \|u\|_V$, the Taylor representation 
$$ 
\begin{aligned} 
G(t, u) & = f(u + v^{TW} (t) - f(v^{TW} (t)) \\ 
& = f^\prime (v^{TW} ) (t) u +  \frac 12 f^{(2)} (v^{TW}(t)) u^2 +  \frac 16 f^{(3)} (v^{TW}(t)) u^3
\end{aligned} 
$$ 
and uniform bounds on $\|f^{(k)} (v^{TW} (t))\|_\infty$, $k=1, 2,3$, we have for $w\in V$ that 
$$
\begin{aligned} 
\langle G(t, u), w\rangle & \le \int |f(u+ v^{TW} (t))-f(v^{TW} (t))||w|\, dx \\ 
& \le c_1 \|u\|_V \left( 1 + \|u\|_H^2\right)\|w\|_V\, , 
\end{aligned} 
$$ 
hence 
\begin{equation} 
\label{Boundedness} 
\|G(t, u)\|_{V^\prime} \le c_1 \|u\|_V \left( 1 + \|u\|_H^2\right) 
\end{equation} 
and 
$$ 
\begin{aligned} 
\langle G(t, u_1) - G(t, u_2), w\rangle & \le \int |f(u_1+ v^{TW} (t))- f(u_2+ v^{TW} (t)) | |w|\, dx  \\ 
& \le c_2 \left( 1+ \|u_1\|^2_V + \|u_2\|^2_V\right) \|u_1 - u_2\|_H \|w\|_V
\end{aligned} 
$$ 
so that  
\begin{equation} 
\label{Lipschitz} 
\| G(t, u_1) - G(t, u_2)\|_{V^\prime} \le c_2 \left( 1+ \|u_1\|^2_V + \|u_2\|^2_V\right) \|u_1 - u_2\|_H
\end{equation} 
for finite constants $c_1$ and $c_2$ depending on $f_{|[0,1]}$ only.

\medskip 
\noindent 
Note also that the sum $\nu\Delta u + bG(t,u)$ satisfies the (global) monotonicity condition 
\begin{equation} 
\label{Monotonicity} 
\langle \nu\Delta u_1 +  bG(t, u_1) - \nu\Delta u_2 -  bG(t, u_2),u_1 - u_2\rangle  \le b\eta \|u_1 - u_2\|_H^2
\end{equation} 
where $\eta = \sup_{\xi\in\Bbb R} f^\prime (\xi ) = \frac {1-a+a^2}3$, since $(f(s) - f(t))(s-t) \le \eta (s-t)^2$ for all 
$s,t\in\Bbb R$, and the coercivity condition 
\begin{equation} 
\label{Coercivity} 
\langle \nu\Delta u + bG(t, u),u\rangle  \le - \nu \|u\|_V^2 + (b\eta + \nu )\|u\|_H^2 
\end{equation} 
since $f(s)s = (f(s) - f(0))(s-0) \le \eta s^2$ for all $s\in\Bbb R$.

\medskip 
\noindent 
It is now standard (see, e.g. Theorem 1.1 in \cite{LR}) to deduce for all $u_0\in H$ and all finite times T existence and uniqueness of a 
variational solution $u\in L^\infty ([0,T]\, ; H)\cap L^2 ([0,T]\, ; V)$ satisfying the integral equation 
\begin{equation} 
\label{integralNag}
u (t) = u_0 + \int_0^t \nu\Delta u(s, \cdot ) + bf(u(s) + v^{TW}(s)) -  bf(v^{TW} (s))\, ds 
\end{equation} 
associated with \eqref{RelNag}. Clearly, we may consider this solution $u$ as a solution on the whole time axes $t\ge 0$.

\medskip 
\noindent 
The integral $\int_0^t \nu\Delta u(s, \cdot ) + bf(u(s) + v^{TW}(s)) - bf(v^{TW} (s))\, ds$ appearing in the integral equation \eqref{integralNag}
is well-defined as a Bochner integral in $L^2([ 0,T]\, ; V^\prime)$, since due to \eqref{Boundedness}
$$ 
\begin{aligned} 
\int_0^t \|\Delta u(s)\|^2_{V^\prime} & + \|f(u(s) + v^{TW}(s)) - f(v^{TW} (s))\|^2_{V^\prime}\, ds\\
& \le c \int_0^t \|u(s)\|^2_V \, ds \, \left( 1 + \sup_{t\in [0,T]} \|u(t)\|_H^4 \right) < \infty 
\end{aligned}
$$ 
for all $t\ge 0$. In particular, the mapping $t\mapsto u(t)$, $[0,\infty )\to V^\prime$, is differentiable with differential 
$$ 
\frac{du(t)}{dt} = \nu\Delta u(t) + bf(u(t)+ v^{TW} (t)) - f(v^{TW} (t)) \in V^\prime 
$$
and therefore also locally Lipschitz.

\begin{definition}
\label{defi1} 
The travelling wave solution $v^{TW}$ is called locally asymptotically stable in $H$ if there exists $\delta > 0$ such that for initial condition 
$v_0$ with $v_0 - v^{TW}\in H$ and $\|v_0-v^{TW}\|_H \le \delta$ the (unique variational) solution $u(t,x)$ to \eqref{RelNag} satisfies 
$$ 
\lim_{t\to\infty} \|v(t,\cdot ) - v^{TW} (t+t_0)\|_H = 0 
$$ 
for some (phase) $t_0\in\Bbb R$. 
\end{definition}

\medskip 
\noindent 
The stability of travelling wave fronts for the Nagumo equation has been studied in many papers as a prototype example for metastability. 
The mathematical analysis of stability properties of $v^{TW}$ faces two major difficulties. The first one is the obvious fact that the reaction 
term $f(u)$ in the equation \eqref{Nag} is not strictly dissipative in the sense that 
$$ 
\langle \nu \partial^2_{xx} v_1 + bf(v_1) - \nu\partial_{xx}^2 v_2 - bf(v_2), v_1 - v_2\rangle \le - \kappa_\ast \|v_1 - v_2\|^2_H
$$ 
for some $\kappa_\ast > 0$, or equivalently, the associated potential $F(v) = \int_{v_0}^v f(t)\, dt$ is not uniformly strictly convex, but a double-well 
potential. This remains true if we fix $v_2$ to be equal to the travelling wave $v^{TW}$ or any of its spatial translates $v^{TW} ( \cdot - y)$. A first naive 
calculation, exploiting the coercivity condition \eqref{Coercivity}, only yields the following a priori estimate.

\medskip 
\noindent 
\begin{lemma}
\label{lem0} 
Let $u\in L^\infty ([0, T]; H)\cap L^2 ([0,T]; V)$ be the unique solution of \eqref{RelNag}. Then 
$$ 
\|u(t)\|^2_H \le e^{2b\eta t} \|u_0\|^2_H \quad \forall t\in [0,T]\, . 
$$ 
\end{lemma}

\begin{proof} 
The coercivity condition \eqref{Coercivity} implies that 
$$ 
\begin{aligned} 
\frac d{dt} \|u(t)\|_H^2 & = 2\langle\nu\Delta u(t) + bG(t, u(t)), u(t)\rangle \\ 
& \le 2b\eta \|u(t)\|^2_H \quad t\in [0, T] \, . 
\end{aligned} 
$$ 
Integrating up the last inequality w.r.t. $t$ yields the desired inequality. 
\end{proof}

\medskip 
\noindent
However, restricting $u$ to the orthogonal component of the derivative $\partial_x v^{TW}$ of the travelling wave solution, i.e. 
$\int u \partial_x v^{TW} \, dx = 0$, we have the following local dissipativity estimate according to the following  

\begin{theorem} 
\label{thm0} 
Let 
$$ 
\kappa_\ast = \frac 25 \frac{\nu b}{\nu + b} (a\wedge (1-a))  
\quad\mbox{ and }\quad 
C_\ast = 6(\nu + b) \, . 
$$  
Then 
$$ 
\langle \nu \partial^2_{xx} u + b\langle f^\prime (v^{TW})u, u\rangle \le - \kappa_\ast \|u\|^2_V  + C_\ast\langle u, \partial_x v^{TW}\rangle^2
$$  
for all $u\in V$. 
\end{theorem} 

\medskip 
\noindent 
The proof of the Theorem is postponed to Section \ref{SectionProofThm0}.

\medskip 
\noindent 
The second difficulty in the mathematical analysis of the stability properties of $v^{TW}$ is to identify the correct phase-shift of $v^{TW} (t + t_0)$ to which to compare 
the given solution $v$ of \eqref{Nag}. To this end we introduce an auxiliary ordinary differential equation of gradient descent type associated with the minimization of the 
distance between $v$ and the set $\cN = \left\{ v^{TW} \left(\cdot + C\right) \mid C\in \Bbb R\right\}$ of all phase-shifted travelling waves. More precisely, given a solution 
$v$ to the Nagumo equation \eqref{Nag} with initial condition $v_0$ satisfying $v_0 - v^{TW} \in H^{1,2} (\Bbb R)$, and given any relaxation rate $m > 0$ (that will be specified 
later) we consider the ordinary differential equation 
\begin{equation} 
\label{GradDesc} 
\begin{aligned} 
\dot{C} (t) & = - m \langle \partial_x v^{TW} (\cdot + C(t)+ct) , v^{TW} (\cdot + C(t)+ct) - v(t, \cdot )\rangle_H \\ 
C(0) & = 0  \, . 
\end{aligned} 
\end{equation} 
The next proposition states that the ordinary differential equation is well-posed.  

\begin{proposition} 
\label{prop1} 
Let $v = u + v^{TW} (t)$ be a solution to \eqref{Nag} with $u\in L^\infty ([0, T];H)\cap L^2 ([0,T]; V)$. Then  
$$
B(t,C) = \langle \partial_x v^{TW} (\cdot + C+ct) , v^{TW} (\cdot + C+ct) - v(t, \cdot )\rangle_H
$$  
is continuous in $(t,C)\in [0, T]\times \Bbb R$, and Lipschitz continuous w.r.t. $C$ with Lipschitz constant independent of $t$. 
\end{proposition}

\begin{proof} 
Using the representation 
$$ 
\begin{aligned} 
B(t,C) & = \langle\partial_x v^{TW} (\cdot +C + ct) , v^{TW} (\cdot +C +ct) - v^{TW} (t)\rangle_H \\ 
& \quad + \langle\partial_x v^{TW} (\cdot + C+ct) , v^{TW} (\cdot + ct) - v(t, \cdot )\rangle_H
\end{aligned} 
$$ 
the continuity of $B$ follows from the continuity of $(t,C)\mapsto \partial_x v^{TW} (\cdot +C + ct)$ 
as a mapping with values in $V$, the continuity of $(t,C)\mapsto v^{TW} (\cdot +C + ct) - v^{TW} (t)$ as a mapping with values in $V^\prime$  
and the (Lipschitz) continuity of $t\mapsto v(t, \cdot) - v^{TW} (t)$ as a mapping with values in $V^\prime$.

\medskip 
\noindent 
For the proof of the Lipschitz property w.r.t. $C$ note that 
$$ 
\begin{aligned} 
B(t, & C_1) -  B(t, C_2) \\
& = \langle\partial_x v^{TW} (\cdot +C_1 + ct) - \partial_x v^{TW} (\cdot +C_2 + ct), v^{TW} (\cdot + ct) - v (t, \cdot )\rangle_H\\ 
& + \langle\partial_x v^{TW} (\cdot +C_1 + ct), v^{TW} (\cdot + C_1 + ct) - v^{TW} (\cdot + ct)\rangle_H \\ 
& - \langle\partial_x v^{TW} (\cdot +C_2 + ct), v^{TW} (\cdot + C_2 + ct) - v^{TW} (\cdot + ct)\rangle_H \\ 
& = I + II + III, \mbox{ say.} 
\end{aligned} 
$$ 
We next assume w.l.o.g. $C_1 \le C_2$. According to the explicit representation \eqref{TWsolution} it follows that 
\begin{equation} 
\label{EstProp1_1} 
\begin{aligned}
| v^{TW} (x +C_1 + ct) & - v^{TW} (x +C_2 + ct) | \\
& = \sqrt{\frac b{2\nu}} \int_{C_1}^{C_2} \frac{\exp (-\sqrt{\frac b{2\nu}} (x+\xi +ct)}{(1 + \exp (-\sqrt{\frac b{2\nu}} (x+\xi +ct)))^2}\, d\xi \\
& = \sqrt{\frac b{2\nu}} \int_{C_1}^{C_2} \exp (-\sqrt{\frac b{2\nu}} (x+\xi +ct)) v^{TW} (x +\xi + ct)^2 \, d\xi \, , 
\end{aligned} 
\end{equation} 
so that we can further estimate 
\begin{equation} 
\label{EstProp1_2}
\begin{aligned} 
|II + III| & = |\langle\partial_x v^{TW} (\cdot + ct), v^{TW} (\cdot + ct) - v^{TW} (\cdot  - C_1 + ct)\rangle_H \\ 
& \quad - \langle\partial_x v^{TW} (\cdot + ct), v^{TW} (\cdot + ct) - v^{TW} (\cdot - C_2 + ct)\rangle_H | \\ 
& = |\langle\partial_x v^{TW} (\cdot + ct), v^{TW} (\cdot - C_1 + ct) - v^{TW} (\cdot - C_2 + ct)\rangle_H | \\ 
& \le \|\partial_x v^{TW} (\cdot + ct)\|_H \\ 
& \qquad \cdot \sqrt{\frac b{2\nu}} \int_{C_1}^{C_2} \|\exp (-\sqrt{\frac b{2\nu}} (\cdot +\xi + ct)) v^{TW} (\cdot  +\xi + ct)^2\|_H\, d\xi \\ 
& \le \mbox{ const}\cdot |C_1 - C_2|\, . 
\end{aligned}
\end{equation} 
Similarly, using  
$$ 
\left| \partial_{xx}^2 v^{TW} (x)\right| \le 3\sqrt{\frac b{2\nu}} \left|\partial_x v^{TW} (x)\right| 
$$ 
\begin{equation} 
\label{EstProp1_3}
\begin{aligned} 
|I| & \le \int_{C_1}^{C_2} \int_{\Bbb R} \left| \partial_{xx}^2 v^{TW} (x + \xi + ct) \,  u(t,x) \right| \, dx\, d\xi \\ 
& \le 3\sqrt{\frac b{2\nu}} \int_{C_1}^{C_2} \|\partial_x v^{TW} (\cdot  +\xi + ct)\|_H   \|u(t)\|_H \, d\xi \\ 
& \le 3\sqrt{\frac b{2\nu}} |C_1 - C_2| \sup_{t\in [0,T]} \|u(t)\|_H  \, . 
\end{aligned} 
\end{equation} 
Inserting \eqref{EstProp1_2} and \eqref{EstProp1_3} into \eqref{EstProp1_1} yields the desired assertion. 
\end{proof}

\medskip 
\noindent
According to the last Proposition the function $C$ defined by \eqref{GradDesc} is well-defined. As already indicated, $C$ will adapt 
to the correct phase of the $v$ if we choose $m \ge C_\ast$ (cf. Theorem \ref{thm0}) and our aim is to prove in the following that the difference 
\begin{equation} 
\label{TildeU}
\tilde{u} (t) := u(t) + v^{TW} (t) - v^{TW} (\cdot + C(t) + ct) = v(t) - v^{TW} (\cdot + C(t) + ct) 
\end{equation} 
converges to zero as $t\to\infty$ if the initial condition $u_0 = v_0 - v^{TW}$ is sufficiently small in the $H$-norm. In the next Proposition we 
first identify the resulting evolution equation for $\tilde{u}$.

\begin{proposition} 
\label{Prop3} 
Let $u = v - v^{TW}(t) \in L^\infty ([0,T];H)\cap L^2 ([0,T];V)$ be a solution of \eqref{RelNag} and let $\tilde{u}$ be defined by \eqref{TildeU}. 
Then $\tilde{u}\in L^\infty ([0,T];H)\cap L^2 ([0,T];V)$ again and $\tilde{u}$ satisfies the evolution equation 
\begin{equation} 
\label{EvEqTildeU} 
\begin{aligned} 
\frac{d\tilde{u}}{dt}(t)  & = \nu\Delta\tilde{u}(t) + b \tilde{G} (t,\tilde{u}(t)) - \dot{C} (t) \partial_x v^{TW} (\cdot + C(t) + ct)  \\ 
& = \nu\Delta\tilde{u}(t) + b \tilde{G} (t,\tilde{u}(t)) \\ 
& \qquad - m \langle\partial_x v^{TW} (\cdot + C(t) + ct) , \tilde{u} (t)\rangle\partial_x v^{TW} (\cdot + C(t) + ct)   
\end{aligned} 
\end{equation}  
with 
$$ 
\tilde{G} (t,u) = f(u + v^{TW} (\cdot + C(t) + ct)) - f(v^{TW} (\cdot + C(t) + ct))\, . 
$$ 
In particular, 
$$ 
\begin{aligned} 
\frac 12 \frac d{dt} \|\tilde{u}(t)\|^2_H 
& = - \nu\|\partial_x u(t)\|^2_H + b\langle \tilde{G} (t, \tilde{u} (t)), \tilde{u}(t)\rangle \\ 
& \quad - m\langle \partial_x v^{TW} (\cdot + C(t) + ct), \tilde{u} (t)\rangle^2 \, .  
\end{aligned} 
$$ 
\end{proposition} 

\noindent
The proof of the Proposition is an immediate consequence of the properties of $v^{TW}$ and the equations \eqref{RelNag} and \eqref{GradDesc}.

\bigskip 
\noindent 
As usual we will now consider the linearization of the mapping $\tilde{G} (t, u)$ around zero. To simplify notations, let 
$\tilde{v}^{TW} (t) := v^{TW} (\cdot + C(t) + ct)$. Then we can write   
\begin{equation} 
\label{Linearization}  
\tilde{G}(t, u) = f^\prime (\tilde{v}^{TW}(t)) u + \tilde{R}(t,u) 
\end{equation}  
where 
$$ 
\begin{aligned} 
\tilde{R}(t,u) 
& = f(u + \tilde{v}^{TW} (t) ) - f(\tilde{v}^{TW} ) - f^\prime (\tilde{v}^{TW} (t))u \\ 
& = \frac 12 f^{(2)} (\tilde{v}^{TW} (t)) u^2 +  \frac 16 f^{(3)} (\tilde{v}^{TW} (t)) u^3
\end{aligned} 
$$ 
satisfies the estimates 
\begin{equation} 
\label{Remainder1} 
\langle\tilde{R} (t,u), u\rangle \le (4+a)\|u\|^2_H \|u\|_V \le (4+a) \|u\|_H \|u\|^2_V  
\end{equation} 
and 
\begin{equation} 
\label{Remainder2} 
\|\tilde{R}(t,u)\|_{V^\prime} \le (4+a) \|u\|_H^2 (1 + \|u\|_V)  
\end{equation} 

\medskip 
\noindent 
Similar to the classical stability analysis of the Nagumo equation we now use the information on the spectrum of the Schr\"odinger operator 
$\nu\Delta u + b f^\prime (v^{TW})u$ contained in Theorem \ref{thm0} with the above localization to obtain the first local stability result.

\subsection{Main result in the deterministic case}

\begin{theorem} 
\label{th1} 
Recall the definition of $\kappa_\ast$ and $C_\ast$ in Theorem \ref{thm0}. If the initial condition $v_0 = u_0 + v^{TW}$ is close 
to $v^{TW}$ in the sense that 
$$ 
\|u_0\|_H < \delta \frac{\kappa_\ast}{b(4+a)}
$$ 
for some $\delta < 1$ and $v(t) = u(t) + v^{TW} (t)$, where $u(t)$ is the unique solution of \eqref{RelNag}, then 
$$ 
\|v(t) - v^{TW} (\cdot + C(t) +ct)\|_H \le e^{-(1-\delta)\kappa_\ast t} \|v_0 -v^{TW}\|_H\, . 
$$ 
Here, $C(t)$ is the solution of \eqref{GradDesc} with $m\ge C_\ast$. 
\end{theorem}

\begin{proof} 
Let $\tilde{u} (t) := v(t) - \tilde{v}^{TW} (t)$ be as in \eqref{TildeU}. Then Proposition \ref{Prop3} and \eqref{Remainder1} 
imply that 
\begin{equation} 
\label{EstTh1_1} 
\begin{aligned} 
\frac 12 \frac{d}{dt} \|\tilde{u}(t)\|^2_H 
& = \langle \nu \Delta\tilde{u}(t) + bf^\prime (\tilde{v}^{TW}(t)) \tilde{u} (t), \tilde{u}(t)\rangle 
     + b\langle \tilde{R} (t, \tilde{u} (t)) , \tilde{u} (t) \rangle \\ 
& \quad - m\left(\langle \partial_x \tilde{v}^{TW}(t) , \tilde{u} (t)\rangle\right)^2 \\ 
& \le \langle \nu \Delta\tilde{u}(t) + bf^\prime (\tilde{v}^{TW}(t)) \tilde{u} (t), \tilde{u}(t)\rangle \\
& \quad + b(4+a) \|\tilde{u} (t) \|_H \|\tilde{u} (t)\|_V^2 - m \left(\langle \partial_x \tilde{v}^{TW}(t) , \tilde{u} (t)\rangle\right)^2\, . 
\end{aligned} 
\end{equation} 
Using translation invariance of $\nu \Delta$ and $\int \left( \partial_x u\right)^2\, dx$, Theorem \ref{thm0} yields the estimate  
\begin{equation} 
\label{EstTh1_2} 
\begin{aligned} 
\langle \nu \Delta\tilde{u}(t) & + bf^\prime (\tilde{v}^{TW}(t)) \tilde{u} (t), \tilde{u}(t)\rangle \\ 
& \le - \kappa_\ast \|\tilde{u} (t)\|_V^2 + C_\ast \left( \int\tilde{u}(t) \partial_x \tilde{v}^{TW} \, dx\right)^2 \, .   
\end{aligned} 
\end{equation} 
Inserting \eqref{EstTh1_2} into \eqref{EstTh1_1} yields that 
$$ 
\begin{aligned} 
\frac 12 \frac d{dt} \|\tilde{u} (t) \|_H^2 
& \le - \kappa_\ast \|\tilde{u} (t) \|^2_V  + b(4+a) \|\tilde{u} (t)\|_H \|\tilde{u} (t) \|_V^2 \, . 
\end{aligned} 
$$ 

\medskip 
\noindent 
In the next step we define the stopping time 
$$ 
T := \inf \left\{ t\ge 0 \mid \|\tilde{u}(t)\|_H \ge \delta \frac{\kappa_\ast}{b(4+a)}\right\}  
$$ 
with the usual convention $\inf \emptyset = \infty$. Continuity of $t\mapsto \|\tilde{u}(t)\|_H$ implies that 
$T > 0$ since $\|u_0\|_H < \delta \frac{\kappa_\ast}{b(4+a)}$. For $t < T$ note that 
$$ 
\frac 12 \frac d{dt} \|\tilde{u} (t) \|_H^2 \le - (1- \delta )\kappa_\ast \|\tilde{u} (t) \|^2_V  
\le - (1- \delta )\kappa_\ast \|\tilde{u} (t) \|^2_H 
$$ 
which implies that 
$$ 
\|\tilde{u}(t)\|^2_H \le e^{-2(1-\delta )\kappa_\ast t} \|u_0\|^2_H  
$$ 
for $t < T$. Suppose now that $T < \infty$. Then continuity of $t\mapsto \|\tilde{u}(t)\|_H$ implies on the one hand 
that $\|\tilde{u}(T)\|_H = \delta \frac{\kappa_\ast}{b(4+a)}$ and on the other hand, using the last inequality,   
$$ 
\|\tilde{u}(T)\|_H = \lim_{t\uparrow T} \|\tilde{u} (t)\|_H \le e^{-(1-\delta )\kappa_\ast T} \|u_0\|_H < \delta \frac{\kappa_\ast}{b(4+a)} 
$$ 
which is a contradiction. Consequently, $T = \infty$ and thus 
$$
\|\tilde{u} (t)\|_H \le e^{-(1-\delta )\kappa_\ast t} \|u_0\|_H \qquad\forall t\ge 0 
$$ 
which implies the assertion. 
\end{proof}

\section{Stochastic stability} 
\label{sec2}

\noindent 
We now turn to the stochastic Nagumo equation 
\begin{equation} 
\label{StochNag} 
\begin{aligned} 
dv (t) & = \left[ \nu\partial^2_{xx} v (t) +  bf(v(t))\right]\, dt  + \sigma (v(t))\, dW^Q (t) 
\end{aligned} 
\end{equation} 
where $\sigma : \R \to\R$ and $W^Q = (W^Q (t))_{t\ge 0}$ is a $Q$-Wiener process on $H$ defined on some underlying filtered probability 
space $(\Omega , \cF , (\cF (t))_{t\ge 0}, P)$. We make the following two assumptions: 
\begin{equation} 
\label{Dispersion} 
\sigma \mbox{ is Lipschitz continuous and } \sigma (0) = \sigma (1) = 0\, . 
\end{equation}
Denote with Lip$_{\sigma}$ its Lipschitz constant. As the covariance operator $Q$ is of trace class, positive semi-definite and symmetric, it has 
a positive semi-definite square root $\sqrt{Q}$ of Hilbert-Schmidt type. If we denote the representing integral kernel with $k_{\sqrt{Q}} (x,y) \in L^2 (\R^2 )$ 
we assume that  
\begin{equation} 
\label{CovarianceOperator} 
M_{\sqrt{Q}} := \sup_{x\in\R} \int k_{\sqrt{Q}} (x,y)^2\, dy < \infty \, . 
\end{equation} 
The theory of Wiener processes on Hilbert spaces and associated stochastic evolution equations can be found in the 
monograph \cite{PR}. 

\medskip 
\noindent 
As in the deterministic case we will give the equation a rigorous meaning by decomposing $v(t) = u(t) + v^{TW} (t)$ 
w.r.t. the (deterministic) travelling wave \eqref{TWsolution}. The stochastic evolution equation for $u$ is then given by 
\begin{equation} 
\label{StochRelNag} 
\begin{aligned} 
du(t) & = \left[ \nu\Delta u (t) + bG(t, u(t))\right]\, dt + \Sigma (t , u (t))\, dW(t) 
\end{aligned} 
\end{equation} 
where the nonlinear term $G$ is as in \eqref{G},  
\begin{equation} 
\label{RealizationDispersion}
\Sigma (t,u)h := \sigma \left( u + v^{TW} (t) \right) \sqrt{Q} h\, , \quad u,h\in H\, ,   
\end{equation} 
is a continuous mapping 
$$ 
\Sigma (\cdot , \cdot ) : [ 0, \infty ) \times H\to L_2 (H) 
$$ 
(where $L_2 (H)$ is the space of Hilbert-Schmidt operators on $H$) 
and $W = (W(t))_{t\ge 0}$ now denotes a cylindrical Wiener process on $H$. 
Note that the two conditions \eqref{Dispersion} and \eqref{CovarianceOperator} now imply, as we show below, that 
\begin{equation}
\label{Dispersion1} 
\| \Sigma (t, u_1 ) - \Sigma (t, u_2)\|^2_{L_2 (H)} \le \mbox{ Lip}^2_\sigma M_{\sqrt{Q}} \|u_1 - u_2\|^2_H 
\end{equation} 
and 
\begin{equation} 
\label{Dispersion2} 
\| \Sigma (t, u)\|^2_{L_2 (H)} \le  2 \mbox{ Lip}^2_\sigma \,  M_{\sqrt{Q}} \left( \|u\|_H^2 + \|v^{TW} \wedge (1-v^{TW})\|_H^2 \right) \, . 
\end{equation} 
Indeed, note that the assumption on $\sqrt{Q}$ implies for any complete orthonormal system $(e_n)_{n\ge 1}$ of $H$ that  
$$ 
\begin{aligned} 
& \|\Sigma (t, u_1 ) - \Sigma (t , u_2) \|^2_{L_2 (H)} \\ 
\qquad & = \sum_{n=1}^\infty \int \left( \left( \sigma ( u_1 + v^{TW} (t) ) - \sigma (u_2 + v^{TW} (t))\right) \sqrt{Q}e_n \right)^2 \, dx \\ 
\qquad & \le \left( \sup_{x} \sum_{n=1}^\infty \sqrt{Q} e_n (x)^2\right)  \int \left( \sigma ( u_1 + v^{TW} (t) ) - \sigma (u_2 + v^{TW} (t)) \right)^2 \, dx \\ 
& \le M_{\sqrt{Q}} \, \mbox{ Lip}^2_\sigma  \|u_1 - u_2\|_H^2 
\end{aligned} 
$$ 
hence the Lipschitz continuity of $\Sigma$ in the Hilbert-Schmidt norm \eqref{Dispersion1} follows.  
Similarly, using the pointwise inequality 
$$
\begin{aligned} 
|\sigma (u(x) + v^{TW}(t,x) ) | & \le \mbox{ Lip}_\sigma ( |u(x)+v^{TW}(t,x)| 1_{\{v^{TW} (t,x)\le \frac 12\}} \\ 
& \qquad + |1-(u(x)+v^{TW}(t,x))|1_{\{v^{TW} (t,x) > \frac 12 \}} ) \\ 
& \le \mbox{ Lip}_\sigma \left( |u(x)| + |v^{TW}(t,x)|\wedge |1- v^{TW} (t,x)|\right) 
\end{aligned} 
$$ 
we obtain that 
\begin{equation} 
\label{EstimateCovariance} 
\begin{aligned} 
\|\Sigma (t, u)\|^2_{L_2 (H)} 
& = \sum_{n=1}^\infty \int \left( \sigma ( u + v^{TW} (t) ) \sqrt{Q}e_n \right)^2 \, dx \\ 
& \le \left( \sup_{x} \sum_{n=1}^\infty \sqrt{Q} e_n (x)^2\right)  \int \sigma ( u + v^{TW} (t) )^2 \, dx \\ 
& \le 2 M_{\sqrt{Q}} \, \mbox{ Lip}^2_\sigma  \left( \|u\|_H^2 + \|v^{TW} (t)\wedge (1-v^{TW} (t))\|_H^2 \right) \\ 
& =  2 M_{\sqrt{Q}} \, \mbox{ Lip}^2_\sigma  \left( \|u\|_H^2 + \|v^{TW}\wedge (1-v^{TW})\|_H^2 \right)  
\end{aligned} 
\end{equation}
hence \eqref{Dispersion2} follows. 

\medskip 
\noindent 
We now consider the equation \eqref{StochRelNag} w.r.t. the same triple $V\hookrightarrow H \equiv H^\prime\hookrightarrow V^\prime$ 
as in Section \ref{sec1}. Due to the properties \eqref{Boundedness}, \eqref{Lipschitz}, \eqref{Monotonicity} and \eqref{Coercivity}, 
we can deduce from Theorem 1.1. in \cite{LR} for all finite $T$ and all (deterministic) initial conditions $u_0\in H$ the 
existence and uniqueness of a solution $(u(t))_{t\in [0, T]}$ of \eqref{StochRelNag} satisfying the moment estimate 
$$ 
E\left( \sup_{t\in[0,T]} \|u(t)\|^2_H + \int_0^T \|u(t)\|_V^2\, dt \right) < \infty 
$$ 
which implies in particular that $u\in L^\infty ([0, T];H)\cap L^2 ([0,T]; V)$ P-a.s. As a consequence we can apply Proposition 
\ref{prop1} to a typical trajectory $u(\cdot )(\omega )$ to obtain a unique solution $C(\cdot )(\omega )$ of equation \eqref{GradDesc}. 
It is also clear that the resulting stochastic process $(C(t))_{t\ge 0}$ is $(\cF_t )_{t\ge 0}$-adapted, since $(u(t))_{t\ge 0}$ is. 
We will assume as in the deterministic case that the relaxation rate $m$ is sufficiently large, i.e., $m > C_\ast$. 

\medskip 
\noindent 
Similar to the deterministic case we now define the stochastic process 
$$ 
\tilde{u} (t) = u(t) + v^{TW} (t) - v^{TW}(\cdot + C(t) + ct)  = v(t) -\tilde{v}^{TW} (t)  
$$ 
which is $(\cF_t)_{t\ge 0}$ adapted too and satisfies the stochastic evolution equation 
$$ 
d\tilde{u} (t) = \left[ \nu\Delta \tilde{u} (t) + b \tilde{G} (t, \tilde{u}(t))  
- \dot{C} (t) \partial_x \tilde{v}^{TW} (t)\right] \, dt + \tilde{\Sigma} (t, \tilde{u} (t))\, dW(t) 
$$
where 
$$ 
\tilde{G} (t,u) = f(u + \tilde{v}^{TW} (t) ) -  f(\tilde{v}^{TW} (t)) \, , 
\tilde{\Sigma} (t,u) = \Sigma (t, u + \tilde{v}^{TW} (t)) 
$$ 
and the moment estimates 
$$ 
E\left( \sup_{t\in [0,T]} \|\tilde{u} (t)\|_H^2 + \int_0^T \|\tilde{u} (t)\|^2_V\, dt \right) < \infty\, . 
$$ 
Due to \cite{PR}, Theorem 4.2.5, we have the Ito-formula 
$$ 
\begin{aligned} 
\|\tilde{u}\|^2_H (t) 
& = \|\tilde{u}(0)\|^2_H + \int_0^t 2\langle \nu\Delta \tilde{u} (s) + \tilde{G}(s, \tilde{u} (s)) \\
& - \dot{C} (s) \partial_x\tilde{v}^{TW} (s), \tilde{u} (s) \rangle  + \|\tilde{\Sigma} (s,\tilde{u} (s))\|^2_{L_2 (H)}\, ds 
+ \tilde{M}_t 
\end{aligned} 
$$ 
with 
$$ 
\tilde{M}_t = 2 \int_0^t \langle \tilde{u} (s),  \tilde{\Sigma} \left( s, \tilde{u} (s) \right) dW(s)\rangle \, . 
$$
It follows from the above representation that $\|\tilde{u}(t)\|_H^2$ is a (scalar-valued) continuous local semimartingale, 
in particular we have also the (one-dimensional) time-dependent Ito-formula 
\begin{equation} 
\label{ItoFormula} 
\begin{aligned} 
\varphi (t,\|\tilde{u}(t)\|^2_H ) 
& = \int_0^t \partial_t \varphi (s, \|\tilde{u} (s)\|^2_H) + 2\partial_x\varphi (s,\|\tilde{u} (s)\|^2_H) 
\langle \nu\Delta \tilde{u} (s)   \\
& \quad + b\tilde{G}(s, \tilde{u} (s)) - \dot{C} (s) \partial_x\tilde{v}^{TW} (s), \tilde{u} (s) \rangle  \\
& \quad + \partial_x\varphi (s,\|\tilde{u} (s)\|^2_H)\|\tilde{\Sigma} (s,\tilde{u} (s))\|^2_{L_2(H)} \\ 
& \quad + \partial^2_{xx} \varphi(s,\|\tilde{u} (s)\|_H^2) 2\|\tilde{\Sigma}^\ast (s,\tilde{u} (s))\tilde{u} (s)\|_H^2\, ds \\
& \quad  + \int_0^t \partial_x \varphi (s, \|\tilde{u} (s)\|^2_H )\, d\tilde{M}_s 
\end{aligned} 
\end{equation} 
for any $\varphi\in C^{1,2}([0,T]\times \Bbb R_+)$. Here, $\tilde{\Sigma}^\ast (s,u)$ denotes the adjoint operator of $\tilde{\Sigma} (s,u)$.

\begin{theorem} 
\label{th2} 
Recall the definition of $\kappa_\ast$ and $C_\ast$ in Theorem \ref{thm0} and assume that $M_{\sqrt{Q}} \mbox{ Lip}_{\sigma}^2 \le \frac{\kappa_\ast} 4$. 
Let $v_0 = u_0 + v^{TW}$. Let $v(t) = u(t) + v^{TW} (t)$, where $u(t)$ is the unique solution of the stochastic evolution equation \eqref{StochRelNag} 
and $\tilde{u}(t) = u(t) + v^{TW} (t) - \tilde{v}^{TW} (t)$. Let 
\begin{equation} 
\label{ExitTime} 
T := \inf\{t\ge 0\mid \|\tilde{u} (t) \|_H > c_\ast \} \, , \qquad c_\ast = \frac{\kappa_\ast}{2b(4+a)} \, , 
\end{equation} 
with the usual convention $\inf\emptyset = \infty$. Then 
$$ 
P\left( T < \infty \right) \le \frac 1{c^2_\ast} \left( \|\tilde{u} (0)\|^2_H + \frac{4M_{\sqrt{Q}} \mbox{ Lip}_{\sigma}^2}{\kappa_\ast} \|v^{TW} \wedge (1-v^{TW})\|^2_H \right) 
$$ 
\end{theorem}

\begin{proof} 
Similar to the proof of Theorem \ref{th1} we have the following inequality 
$$
\begin{aligned} 
& \langle \nu\Delta \tilde{u} (t) + b\tilde{G}(t, \tilde{u} (t)) 
- \dot{C} (t) \partial_x\tilde{v}^{TW} (t), \tilde{u} (t) \rangle \\
& 
\qquad \le -\kappa_\ast \|\tilde{u} (t)\|^2_V + b(4+a)\|\tilde{u}(t)\|_H \|\tilde{u}(t)\|^2_V\, . 
\end{aligned} 
$$
In particular,  
$$ 
\langle \nu\Delta \tilde{u} (t) + b\tilde{G}(t, \tilde{u} (t)) 
- \dot{C} (t) \partial_x\tilde{v}^{TW} (t), \tilde{u} (t) \rangle 
\le -\frac{\kappa_\ast}2 \|\tilde{u} (t)\|^2_V 
$$
for $t\le T$, where $T$ is as in \eqref{ExitTime}. Since also 
$$ 
\begin{aligned} 
\|\tilde{\Sigma} (\tilde{u} (t))\|_{L_2 (H)}^2 
& \le 2 M_{\sqrt{Q}} \mbox{ Lip}_{\sigma}^2 \left( \|\tilde{u}(t)\|^2_H + \|v^{TW}\wedge (1-v^{TW})\|^2_H \right) \\ 
& \le \frac{\kappa_\ast}2 \|\tilde{u}(t)\|^2_H + 2 M_{\sqrt{Q}} \mbox{ Lip}_{\sigma}^2 \|v^{TW}\wedge (1-v^{TW})\|^2_H 
\end{aligned} 
$$ 
it follows that 
\begin{equation} 
\label{EstTh2_1}
\begin{aligned} 
& 2\langle \nu\Delta \tilde{u} (t) + b\tilde{G}(t, \tilde{u} (t)) 
- \dot{C} (t) \partial_x\tilde{v}^{TW} (t), \tilde{u} (t) \rangle + \|\tilde{\Sigma} (t,\tilde{u} (t))\|^2_{L_2 (H)} \\ 
\qquad & \le -\frac{\kappa_\ast}2 \|\tilde{u} (t)\|^2_V + 2 M_{\sqrt{Q}}\mbox{ Lip}_{\sigma}^2  \|v^{TW}\wedge (1-v^{TW})\|^2_H \, . 
\end{aligned} 
\end{equation}
Applying Ito's formula \eqref{ItoFormula} to $e^{\frac {\kappa_\ast} 2 t}x$, \eqref{EstTh2_1} implies 
for $t < T$ that 
$$ 
\begin{aligned} 
e^{\frac{\kappa_\ast} 2 t}\|\tilde{u} (t)\|^2_H 
& \le \|\tilde{u}(0)\|^2_H + \frac{4M_{\sqrt{Q}}\mbox{Lip}_{\sigma}^2}{\kappa_\ast} \left( e^{\frac{\kappa_\ast} 2 t} - 1\right) \|v^{TW}\wedge (1-v^{TW})\|_H^2 \\
& \qquad + \int_0^t e^{\frac{\kappa_\ast} 2 s}\, d\tilde{M}_s\, . 
\end{aligned} 
$$ 
Taking expectations we obtain  
$$ 
E\left( \|\tilde{u} (t\wedge T )\|^2_H \right) \le  \|\tilde{u}(0)\|^2_H + \frac{4M_{\sqrt{Q}}\mbox{Lip}_{\sigma}^2}{\kappa_\ast} \|v^{TW}\wedge (1-v^{TW})\|_H^2 
$$ 
and thus in the limit $t\uparrow\infty$ 
$$ 
\begin{aligned} 
c_\ast^2 P\left( T < \infty \right) 
& =  E\left( \|\tilde{u} (T)1_{T < \infty} \|^2_H \right) \le \lim_{t\uparrow\infty} E\left( \|\tilde{u}(t\wedge T) \|^2_H \right)  \\
& \le  \|\tilde{u}(0)\|^2_H + \frac{4M_{\sqrt{Q}}\mbox{Lip}_{\sigma}^2}{\kappa_\ast} \|v^{TW}\wedge (1-v^{TW})\|_H^2 
\end{aligned} 
$$ 
which implies the assertion. 
\end{proof} 

\begin{remark} \rm 
The theorem establishes a global bound on the error between the solution $v$ of the stochastic Nagumo equation \eqref{StochNag} and the 
phase-shifted travelling wave $v^{TW} (\cdot + ct + C(t))$ on the set $T = \infty$. The probability that $T$ is infinite depends on two parameters, 
one is the initial error $\|\tilde{u} (0)\| = \|v - v^{TW}\|$ and the other component depends on the covariance operator of the noise term. 
In particular, the smaller the noise amplitude in the sense that Lip$_{\sigma}$ and/or $M_{\sqrt{Q}}$ are small, the smaller the probability for $T$ 
being finite. In this sense the stochastic process $C(t) + ct$ gives the correct speed of the wave front and we will use the associated random ordinary 
differential equation in future work to study rigorously its statistical properties. 
\end{remark}

\section{Proof of Theorem \ref{thm0} }
\label{SectionProofThm0}

\noindent 
The proof of Theorem \ref{thm0} requires a number of preliminary results. To simplify notations in the following we simply write $v$ instead of  
$v^{TW}$ in the whole section. Let $w(x) = v_x = k v(1-v)(x) = k\frac{e^{-kx}}{(1+e^{-kx})^2}$ with $k= \sqrt{\frac b{2\nu}}$. 

\begin{proposition} 
\label{propGS1} 
Let $u\in C_c^1 (\R)$ and write $u = hw$. Then 
$$ 
\langle \nu u_{xx} + bf^\prime (v) u,u\rangle 
\le - 2a\wedge (1-a) \nu\int h_x^2 w^2\, dx + 6|1-2a|\nu \langle h, w^2\rangle^2 \, . 
$$ 
\end{proposition}

\begin{proof} 
First note that 
$$ 
\nu u_{xx} + bf^\prime (v) u 
= \left( \nu h_{xx} + 2\nu\frac{w_x}{w} h_x  + c\frac{w_x}{w} h\right) w    
$$ 
because 
$$ 
\nu w_{xx} + bf^\prime(v) w   = cw_x \, .
$$ 
Integrating against $u\, dx = hw\, dx$ yields 
$$ 
\begin{aligned} 
\langle \nu u_{xx} + bf^\prime (v) u,u\rangle 
& = \int\left( \nu h_{xx} + 2\nu\frac{w_x}{w} h_x \right) \, h w^2\, dx  +  c\int h^2 w_x w\,dx \\ 
& =  - \nu \int h_x^2 w^2\, dx + c \int h^2 w_x w\, dx \, . 
\end{aligned} 
$$ 
We will prove in Lemma \ref{LemmaPerturbation} below that 
$$ 
\left| \int h^2 w_x w \, dx \right| 
\le \frac 1k \int h_x^2 w^2 \, dx + \frac 6k \left( \int h w^2 \, dx \right)^2  \, . 
$$ 
Using $c = \sqrt{2\nu b} (\frac 12 - a) = k (1-2a) \nu$, we conclude that 
$$
\begin{aligned} 
\langle \nu \Delta u + bf^\prime (v) u,u\rangle 
& \le - 2a\wedge (1-a)  \nu \int h_x^2 w^2\, dx  \\
& \quad + 6|1-2a| \nu \left( \int h w^2 \, dx \right)^2  \, . 
\end{aligned} 
$$
This proves the assertion. 
\end{proof}

\medskip 
\noindent 
The following Lemma has been used in the previous proof.

\begin{lemma} 
\label{LemmaPerturbation} 
Let $h\in C_b^1 (\R)$. Then 
$$ 
\left| \int h^2 w_xw\, dx \right| \le \frac 1k \int h_x^2 w^2\, dx + \frac 6k \left( \int hw^2\, dx\right)^2 \, . 
$$ 
\end{lemma} 

\noindent 
The proof of the Lemma requires additional information on functional inequalities satisfied by the 
gradient form $\int h_x^2\, w^2\, dx$, which will be provided in Proposition \ref{PropPoincare} and 
in Lemma \ref{LemmaHardy2} first:

\begin{proposition} 
\label{PropPoincare}
The following inequality 
\begin{equation} 
\label{Poincare} 
\int  h^2 \, w^2 \, dx 
\le \frac 4{3k^2}  \int h_x^2\, w^2 \, dx + \frac 6k \left( \int h w^2 \, dx \right)^2    
\end{equation} 
holds for all $h\in C_b^1 (\R)$. Here, the constant $\frac 4{3k^2}$ is the best possible. 
\end{proposition} 

\begin{proof} 
We will first show that 
\begin{equation} 
\label{Poincare1} 
\int (h - h(0))^2 \, w^2 \, dx\le \frac{4}{3k^2} \int h_x^2 w^2 \, dx\, . 
\end{equation} 
To this end we will split up the estimate w.r.t. $x\ge 0$ (resp. $x\le 0$) and show that  
\begin{equation} 
\label{Poincare2} 
\int_0^\infty  (h - h(0))^2 \, w^2 \, dx \le \frac{4}{3k^2} \int_0^\infty h_x^2  w^2 \, dx  
\end{equation} 
and 
\begin{equation} 
\label{Poincare3} 
\int_{-\infty}^0 (h - h(0))^2 \, w^2 \, dx \le \frac{4}{3k^2} \int_{-\infty}^0 h_x^2 w^2 \, dx\, . 
\end{equation} 

\medskip 
\noindent 
Indeed note that for $x\ge 0$, using 
$$ 
\begin{aligned} 
(h(x) - h(0))^2 
& = \left( \int_0^x h_x (s) \, ds \right)^2  
\le  \int_0^x w^{-\frac 12}(s) \, ds \int_0^x h_x^2 w^{\frac 12} \, ds \\ 
& = -\frac 2{k^2} w_{x} w^{-\frac 32}(x) \int_0^x h_x^2 w^{\frac 12} \, ds  
\end{aligned} 
$$ 
since 
\begin{equation} 
\label{Poincare4} 
\begin{aligned} 
\frac d{dx} \left( w_{x} w^{-\frac 32}\right) 
& = k \frac d{dx} \left( (1-2v) w w^{-\frac 32}\right) = k \frac d{dx}\left( (1-2v) w^{-\frac 12}\right)  \\ 
& = kw^{-\frac 12} \left( - 2w - \frac k2 (1-2v)^2 \right) = - \frac{k^2}2 w^{-\frac 12}\, . 
\end{aligned} 
\end{equation} 
Integrating the last inequality against $w^2 \, dx$ we obtain that 
$$ 
\begin{aligned} 
\int_0^\infty (h-h(0))^2 w^2 \, dx 
& \le -\frac  2{k^2}  \int_0^\infty  h_x^2 (s) w^{\frac 12} (s) \int_s^\infty w_{x}(x) w^{\frac 12} (x)\, dx \, ds \\ 
& = \frac  4{3k^2} \int_0^\infty h_x^2 (s) w^2 (s) \, ds 
\end{aligned} 
$$ 
which gives \eqref{Poincare2}.  

\medskip 
\noindent 
For the proof of \eqref{Poincare3} note that $w^2(-x) = w^2(x)$, so that \eqref{Poincare2} implies \eqref{Poincare3}. Clearly, 
combining \eqref{Poincare2} and \eqref{Poincare3} implies \eqref{Poincare1}. For the final step of the proof of inequality \eqref{Poincare} 
let us consider the probability measure $\mu (dx) := Z^{-1} w^2 \, dx$, 
where 
$$ 
Z = \int w^2 \, dx = k\int v(1-v)v_x\, dx = k \int_0^1 v(1-v)\, dv = \frac k6 
$$ 
is the normalizing constant. Then 
$$ 
\begin{aligned} 
\int h^2 (x) \, \mu (dx)  & = \mbox{ Var }_{\mu} (h) + \left( \int h\, d\mu \right)^2 \\ 
& \le Z^{-1} \int (h-h(0))^2\, w^2 \, dx + \left( \int h\, d\mu \right)^2\\ 
& \le Z^{-1} \frac 4{3k^2} \int (\partial_x h)^2 \, w^2 \, dx  + \left( \int h\, d\mu \right)^2 
\end{aligned} 
$$ 
which implies the desired inequality. 

\medskip 
\noindent 
To see that the constant $\frac 4{3k^2}$ is the best possible one, consider the function 
$h_0 (x) = v_{xx} v_x^{-\frac 32}$. Clearly, $\int h_0 w^2\, dx = \int v_{xx} v_x^{\frac 12} \, dx = 0$, 
$$ 
\begin{aligned} 
\int h_0^2 w^2\, dx 
& = \int v_{xx}^2 v_x^{-1}\, dx = k^2 \int (1-2v)^2 v_x \, dx \\ 
& = k^2 \int_0^1 (1-2v)^2\, dv = \frac {k^2}{3}  
\end{aligned} 
$$   
and due to \eqref{Poincare4} $h_{0,x} = - \frac{k^2}2 v_x^{-\frac 12}$, hence 
$$ 
\int h_{0,x}^2 w^2\, dx = \frac{k^4}{4} \int v_x\, dx = \frac{k^4}{4} \int_0^1\, dv = \frac{k^4}{4}\, . 
$$ 
Combining all these equalities yields 
$$ 
\int h_0^2\, w^2\, dx  = \frac 4{3k^2} \int h_{0,x}^2 w^2\, dx  + \frac 6k \left( \int h_0 w^2\, dx\right)^2 \, . 
$$
Hence, if $\kappa$ denotes the minimal constant for which the inequality 
$$  
\int h^2\, w^2\, dx  \le \kappa \int h_x^2 w^2\, dx   
$$
holds for any $h\in C_b^1 (\R )$ with $\int h w^2\, dx = 0$, it follows by approximation of $h_0$ and its derivative in $L^2 (w^2\, dx)$ with functions in 
$C_b^1 (\R)$ that the same inequality also holds for $h_0$ which implies $\kappa\ge\frac{k^4}{3k^2}$. 
\end{proof}

\medskip 
\noindent 
The Poincar\'e inequality  proven above will be only sufficient to control the lower order term $c\int h^2 w_x w\, dx$ if the wave 
speed $c$ is sufficiently small which means that $a$ is sufficiently close to $\frac 12$. For small $a$ however, we will need an 
additional information provided by inequalities contained in the following two lemmas: 

\begin{lemma} 
\label{LemmaHardy1}
Let $h\in C_b^1 (\R_+)$. Then 
\begin{equation} 
\label{Hardy} 
\int_0^\infty  h^2 \, w^2 \, dx 
\le \frac 1{k^2}  \int_0^\infty h_x^2\, w^2 \, dx + \frac{12}k \left( \int_0^\infty  h w^2 \, dx \right)^2  \, . 
\end{equation} 
\end{lemma} 

\begin{proof} 
Let $v_\ast = \frac 12 + \frac 1{2\sqrt{3}}$ be the unique solution $v_\ast \in (\frac 12 , 1)$ of $6 v_\ast (1-v_{\ast}) = 1$ and 
let $x_\ast := v^{-1} (v_\ast) > 0$. We will now first show that 
\begin{equation} 
\label{Hardy1} 
\int_{x_\ast}^\infty  (h - h(x_\ast))^2 \, w^2 \, dx\le \frac 1{k^2} \int_{x_\ast}^\infty h_x^2 w^2 \, dx  
\end{equation} 
and 
\begin{equation} 
\label{Hardy2} 
\int_0^{x_\ast}  (h(x_\ast ) - h)^2 \, w^2 \, dx \le \frac 1{k^2} \int_0^{x_\ast} h_x^2 w^2 \, dx  \, . 
\end{equation}

\medskip 
\noindent 
For the proof of both inequalities note that 
$$ 
v_{xxx} = k^2 (1-6v(1-v))v_x
$$ 
and 
$$ 
\frac{d}{dx} \left( \frac{1-6v(1-v)}{v_x}\right) = -\frac{v_{xx}}{v^2_x}\, . 
$$ 
It follows for $x\ge x_\ast$ that 
$$ 
\begin{aligned} 
\left( h(x) - h(x_{\ast}) \right)^2 
&  = \left( \int_{x_\ast}^x h_x(s)\, ds \right)^2 \\ 
& \le \int_{x_\ast}^x h_x^2 (s)\left( - \frac{v_{x}^2}{v_{xx}}\right) (s) \, ds 
       \int_{x_\ast}^x - \frac{v_{xx}}{v_x^2} (s)\, ds \\
& = \int_{x_\ast}^x h_x^2 (s) \left( - \frac{v_x^2}{v_{xx}}\right) (s)\, ds \cdot \frac{1-6v(1-v)(x)}{v_x (x)} \, . 
\end{aligned} 
$$ 
Integrating the last inequality against $w^2\, dx$ yields 
$$ 
\begin{aligned} 
\int_{x_\ast}^\infty & \left( h(x) - h(x_\ast )\right)^2 \, w^2 (x) \, dx \\
& \le \int_{x_\ast}^\infty h_x^2 (s) \left( -\frac{v_x^2}{v_{xx}}(s)\right) \int_s^\infty \left( 1-6v(1-v)(x)\right)\, v_x (x)\, dx \, ds \\ 
& = \frac 1{k^2} \int_{x_\ast}^\infty h_x^2 (s) \left( -\frac{v_x^2}{v_{xx}}\right) (s) \left( -v_{xx} (s)\right) \, ds   
  = \frac 1{k^2} \int_{x_\ast}^\infty h_x^2 w^2\, ds \, . 
\end{aligned} 
$$ 
Similarly, for $x\le x_\ast$ 
$$ 
\left( h(x_\ast ) - h(x) \right)^2 
\le \int_x^{x_\ast} h_x^2 (s) \left( - \frac{v_x^2}{v_{xx}}\right) (s)\, ds \cdot \frac{6v(1-v)(x)-1}{v_x (x)}  
$$ 
and integrating the last inequality against $w^2\, dx$ again yields 
$$ 
\begin{aligned} 
\int_0^{x_\ast} & \left( h(x_\ast ) - h(x)\right)^2 \, w^2 (x) \, dx  \\ 
& \le \int_0^{x_\ast} h_x^2 (s) \left( -\frac{v_x^2}{v_{xx}}\right)(s) \int_0^s \left( 6v(1-v)(x)-1\right)\, v_x (x)\, dx \, ds \\ 
& = \frac 1{k^2} \int_0^{x_\ast} h_x^2 w^2\, ds \, . 
\end{aligned} 
$$ 

\medskip 
\noindent 
Combining \eqref{Hardy1} and \eqref{Hardy2} we obtain the inequality 
\begin{equation} 
\label{Hardy3} 
\int_0^\infty \left( h(x) - h(x_\ast )\right)^2 \, w^2 (x) \, dx \le \frac 1{k^2} \int_0^\infty h_x^2\, w^2\, dx \, . 
\end{equation} 
 
\medskip 
\noindent 
For the final step let us consider the probability measure $\mu (dx) := Z^{-1} w^2 \, dx$ on $\R_+$, where 
$$ 
Z = \int_0^\infty w^2\, dx = k \int_0^\infty v(1-v) v_x\, dx = k\int_{\frac 12}^1 v(1-v)\, dv =  \frac{k}{12} 
$$ 
is a normalizing constant. Then \eqref{Hardy3} implies 
$$ 
\begin{aligned} 
\int_0^\infty h^2 (x) \, \mu (dx)  & = \mbox{ Var }_{\mu} (h) + \left( \int_0^\infty h\, d\mu \right)^2 \\ 
& \le Z^{-1} \int_0^\infty (h-h(x_\ast ))^2\, w^2 \, dx + \left( \int_0^\infty h\, d\mu \right)^2 \\ 
& \le Z^{-1} \frac 1{k^2}  \int_0^\infty (\partial_x h)^2 \, w^2 \, dx  + \left( \int_0^\infty h\, d\mu \right)^2 
\end{aligned} 
$$ 
which implies the assertion 
$$ 
\int_0^\infty h^2 w^2\, dx \le \frac 1{k^2} \int_0^\infty h_x^2\, w^2\, dx + \frac {12}{k} \left( \int_0^\infty hw^2\, dx\right)^2 \, . 
$$ 
\end{proof}

\begin{lemma} 
\label{LemmaHardy2}
Let $h\in C_b^1 (\R_+ )$ be such that $h(0)=0$. Then 
\begin{equation} 
\label{Hardy2_1} 
\int_0^\infty  h^2 \, w_x^2 \, dx \le \int_0^\infty h_x^2\, w^2 \, dx  \, . 
\end{equation} 
\end{lemma} 

\begin{proof} 
For the proof of the inequality note that 
$$  
\frac{d}{dx} \left( -\frac{v_{xx}}{v^2_x} \right) = k \frac{1-2v(1-v)}{v(1-v)} = k^2 \frac{1-2v(1-v)}{w}\, . 
$$ 
It follows for $x\ge 0$ that 
$$ 
\begin{aligned} 
h(x)^2 
&  = \left( \int_0^x h_x \, ds \right)^2 \\ 
& \le \int_0^x h_x^2 \frac{w}{k^2 (1-2v(1-v))} \, ds \int_0^x \frac{k(1-2v(1-v))}{v(1-v)}\, ds \\ 
& = \int_0^x h_x^2 \frac{w}{k^2 (1-2v(1-v))} \, ds \left( -\frac{w_x}{w^2} (x) \right) \, . 
\end{aligned} 
$$ 
Integrating the last inequality against $w^2\, dx$ yields 
$$ 
\begin{aligned} 
\int_0^\infty h^2 \, w^2 \, dx & 
  \le \int_0^\infty h_x^2 \frac{w}{k^2 (1-2v(1-v))} \int_s^\infty \left( -\frac{w_x}{w^2}\right) w_x^2 \, dx\, ds \\ 
& =  \int_0^\infty h_x^2 w^2 \, ds 
\end{aligned} 
$$ 
using 
$$ 
\begin{aligned} 
- \int_s^\infty \frac{w_x}{w^2} w^2_x \, dx 
& - \int_s^\infty \frac{w_x^3}{w^2}\, dx = k^3 \int_s^\infty (1-2v)^3 w\, dx \\ 
& = k^3 \int_{v(s)}^1 (1-2v)^3\, dv = \frac{k^3}{8} \left( 1- (1-2v(s))^4 \right) \\ 
& = k^3 v(1-v)(1-2v(1-v))\, . 
\end{aligned} 
$$ 
\end{proof} 

\medskip 
\noindent 
Due to symmetry the previous lemma also implies that  
$$ 
\int_{-\infty}^0  h^2 \, w_x^2 \, dx \le \int_{-\infty}^0 h_x^2\, w^2 \, dx 
$$  
for $h\in C_b^1 (\R_- )$ with $h(0) = 0$, hence 
\begin{equation} 
\label{Hardy2_2} 
\int_{-\infty}^\infty  \left( h - h(0)\right)^2 \, w_x^2 \, dx \le \int_{-\infty}^\infty h_x^2\, w^2 \, dx   
\end{equation} 
for $h\in C_b^1 (\R )$. 

\medskip 
\noindent 
We can now turn back to the proof of Lemma \ref{LemmaPerturbation}. 

\begin{proof} (of Lemma \ref{LemmaPerturbation}) 
Let us denote with $\tilde{h} (x) = \frac 12\left( h(x) + h(-x)\right)$ (resp. $\hat{h} (x) =\frac 12 \left( h(x) - h(-x)\right)$) 
the even (resp. odd) part of $h$. Then 
\begin{equation} 
\label{pert0} 
\begin{aligned} 
\left|\int_\R h^2 w_x w\, dx \right| & =  \left|\int_\R (\tilde{h} + \hat{h})^2 w_x w \, dx\right| = 2 \left|\int_\R\tilde{h} \hat{h} w_x w\, dx \right|  \\ 
& \le k\int_\R \tilde{h}^2 w^2\, dx + \frac 1k \int_\R \hat{h}^2 w_x^2\, dx 
\end{aligned} 
\end{equation} 
using $\int_\R \tilde{h}^2 w_x w\, dx = \int_\R \hat{h}^2 w_x w\, dx = 0$. The previous Lemma \ref{LemmaHardy2} and 
$\int_\R \tilde{h}w^2\, dx = \frac 12 \int_\R hw^2\, dx$ imply that 
\begin{equation} 
\label{pert1} 
\begin{aligned} 
\int\tilde{h}^2 w^2\, dx 
& = 2 \int_0^\infty \tilde{h}^2 w^2\, dx \le \frac 2{k^2} \int_0^\infty \tilde{h} w^2\, dx 
     + \frac{24}{k} \left( \int_0^\infty \tilde{h} w^2\, dx \right)^2 \\ 
& = \frac 1{k^2} \int_\R  \tilde{h}_x^2 w^2\, dx + \frac 6k \left( \int_\R hw^2\, dx\right)^2 \, . 
\end{aligned} 
\end{equation} 
Similarly, equation \eqref{Hardy2_2} and $\hat{h}(0) = 0$ imply that 
\begin{equation} 
\label{pert2} 
\int_\R \hat{h}^2 w_x^2\, dx \le \int_\R \hat{h}_x^2 w^2\, dx\, . 
\end{equation} 
Inserting \eqref{pert1} and \eqref{pert2} into \eqref{pert0} yields 
$$ 
\begin{aligned} 
\left| \int_\R h^2 w_x w \, dx \right| 
& \le \frac 1k \int_\R \left( \tilde{h}^2_x + \hat{h}^2_x\right)\, w^2\,dx + \frac 6k \left( \int_\R hw^2\, dx\right)^2 \\ 
& = \frac 1k \int_\R h_x^2 w^2\, dx + \frac 6k \left( \int_\R hw^2\, dx\right)^2 \, . 
\end{aligned} 
$$ 
\end{proof}

\begin{lemma} 
\label{LemmaNormEquivalence} 
Let $u \in C_c^1 (\R)$ and write $u = hw$. Then 
$$ 
\int u_x^2 \, dx + \int u^2\, dx 
\le q_1 \int h_x^2 w^2\, dx + q_2 \langle u, v_x\rangle^2  
$$
with 
\begin{equation} 
\label{Constants} 
q_1 := 5 \left( 1 + \frac {\nu}b \right)   
\qquad\mbox{ and }\qquad  
q_2 :=  18 \sqrt{\frac{2}{\nu b}} (\nu + b)\, . 
\end{equation} 
\end{lemma}

\begin{proof} 
Clearly, 
$$ 
\begin{aligned} 
\nu \int u_x^2\, dx 
& = \nu \int h_x^2 w^2 \, dx + 2\nu\int h_x h w_x w \, dx + \nu \int w_x^2 h^2 \, dx \\
& = \nu \int h_x^2 w^2 \, dx - \nu\int h^2 w_{xx} w\, dx \, . 
\end{aligned} 
$$ 
Using the fact that 
$$ 
- \nu w_{xx}  = bf^\prime (v) w - cw_x \le \left( b\eta + ck\right) w  
$$ 
and the obvious estimate $ck = b \left( \frac 12 -a\right) \le b$, we obtain that 
$$ 
\nu \int u_x^2\, dx \le \nu \int h_x^2 w^2\, dx + b(\eta + 1) \int h^2 w^2\, dx \, . 
$$ 
Using the Poincar\'e inequality \eqref{Poincare} again, as well as $q := \nu + b(\eta + 1)$  we arrive at 
$$ 
\nu \int u^2 + u_x^2 \, dx \le \left( \nu  + q \frac{4}{3k^2} \right) \int h_x^2 w^2\, dx 
+ q\frac{12}{k}\langle u, w\rangle^2 
$$ 
which implies the desired inequality, using the inequalities 
$$ 
\nu + q\frac{4}{3k^2} = \nu + \left( \nu + b(\eta + 1) \right) \frac 83 \frac \nu b \le \nu 5 \left( 1 + \frac {\nu}b \right) 
$$ 
and 
$$ 
q\frac {12}k = (\nu + b(\eta + 1)) 12 \sqrt{\frac{2\nu}b} \le \nu 18 \sqrt{\frac 2{\nu b}} (\nu + b) \, . 
$$ 
\end{proof} 

\medskip 
\noindent 
\begin{proof} (of Theorem \ref{thm0})
First let $u\in C_c^1 (\R)$. Then Proposition \ref{propGS1} implies the estimate 
$$ 
\begin{aligned} 
\langle \nu \Delta u + bf^\prime (v)u,u\rangle  
& \le - 2a\wedge (1-a)  \nu \int h_x^2  w^2 \, dx + 6|1-2a|\nu \langle u, v_x\rangle^2 \, . 
\end{aligned} 
$$ 
Combining the last estimate with the previous Lemma \ref{LemmaNormEquivalence}, we obtain that 
$$ 
\begin{aligned} 
\qquad & \langle \nu \Delta u + bf^\prime (v)u,u\rangle  \le - \frac{2a\wedge (1-a)\nu}{q_1} \|u\|^2_V \\
& \qquad + \left( 6|1-2a|\nu + 2a\wedge (1-a) \nu \frac{q_2}{q_1}\right) \langle u, v_x\rangle^2  \\ 
& \le - \frac 25 \frac{\nu b}{\nu + b} a\wedge (1-a) \|u\|_V^2 \\ 
& \qquad + \left( 6|1-2a| \nu+ 6 a\wedge (1-a) \sqrt{2b\nu} \right) \langle u, v_x\rangle^2 \\ 
& \le - \frac 25 \frac{\nu b}{\nu + b} a\wedge (1-a) \|u\|_V^2 + 6(\nu + b) \langle u, v_x\rangle^2 \, . 
\end{aligned} 
$$ 
This implies Theorem \ref{thm0} with 
$$ 
\kappa_\ast = \frac 25 \frac{\nu b}{\nu + b}a 
$$ 
and 
$$ 
C_\ast = 6 (\nu + b)\, .   
$$ 
For general $u\in V$ we deduce the desired estimate by standard approximation of $u$ with $C_c^1$-functions in the $V$-norm. 
\end{proof}

\bigskip 
\noindent 
{\bf Acknowlegdement} This work is supported by the BMBF, FKZ 01GQ1001B.

\end{document}